\documentclass[10pt]{amsart}

\usepackage{amssymb}
\usepackage{amsthm}
\usepackage{amsmath}
\usepackage[usenames]{color}
\numberwithin{equation}{section}

\title[The inverse function theorem]{The inverse function theorem and the resolution of the Jacobian conjecture in free analysis}
\author{
J. E. Pascoe
}
\begin{document}

\bibliographystyle{plain}

\def\m{\mathfrak{m}}
\def\g{\tilde{g}}
\def\B{\mathcal{B}}
\def\K{\mathcal{K}}
\def\M{\mathcal{M}}
\def\O{\mathcal{O}}
\def\P{\mathcal{P}}
\def\S{\mathcal{S}}
\def\U{\mathcal{U}}
\def\V{\mathcal{V}}
\def\W{\mathcal{W}}
\def\GL{\mathcal{G}\mathcal{L}}

\newtheorem{definition}[equation]{Definition}
\newtheorem{lemma}[equation]{Lemma}
\newtheorem{prop}[equation]{Proposition}
\newtheorem{thm}[equation]{Theorem}
\newtheorem{claim}[equation]{Claim}
\newtheorem{cor}[equation]{Corollary}
\newtheorem{ques}[equation]{Question}
\newtheorem{obs}[equation]{Observation}
%%%%%%%%%%%%%
% Section 1 %
%%%%%%%%%%%%%

\begin{abstract}
We establish an invertibility criterion for free polynomials and free functions
evaluated on some tuples of matrices.
We show that if the derivative is nonsingular on some domain
closed with respect to direct sums and similarity, the function must be invertible.
Thus, as a corollary, we establish the Jacobian conjecture in this context.
Furthermore, our result holds for commutative polynomials evaluated on tuples of commuting matrices.
\end{abstract}
\maketitle

%\tableofcontents

\section{Introduction}
	A \emph{free map} is a function defined on some structured subset of tuples of matrices that respects joint invariance.
	For the purposes of introduction, the standard example is a free polynomial being evaluated on tuples of matrices.
	We give a formal defintion in Section \ref{freeanalysis}.
	
	We consider the inverse function theorem for free maps.
	
	Classically, the inverse function theorem states
	that given a map $f$, if $Df(x)$ is nonsingular for some $x,$
	then there is a neighborhood $U$ of $x$ such that $f^{-1}$ is well-defined on $f(U).$
	
	Our inverse function theorem is as follows.
	\begin{thm}[Inverse function theorem]
		Let $f$ be a free map. The following are equivalent:
		\begin{enumerate}
			\item $Df(X)$ is a nonsingular map for every $X.$
			\item $f$ is injective.
			\item $f^{-1}$ exists and is a free map.
		\end{enumerate}
	\end{thm}
	In the classical case, one may obtain a neighborhood of $x$ where the derivative is nonsingular.
	The geometry of free analysis is not exactly topological due noted algebraic obstructions such as those observed in the work of
	D. S. Kaliuzhnyi-Verbovetskyi and V. Vinnikov \cite{vvw}, and
	 Amitsur and Levitzki \cite{al50}.
	Thus, we assert nonsingularity of the derivative on the entire domain in our theorem, and, in return, we obtain a global result.
	We prove the inverse function theorem result in Section \ref{ifts}.

	We also consider a famous conjecture of Ott-Heinrich Keller, the so-called Jacobian conjecture \cite{oh39}.
	\begin{ques}
		Let $P:\mathbb{C}^N \rightarrow \mathbb{C}^N$ be a polynomial map. If the Jacobian $DP(x)$ is invertible for every $x\in \mathbb{C}^N,$
		is the map $P$ itself invertible?
	\end{ques}

	James Ax \cite{ax68} and Alexander Grothendieck \cite{ega} independently showed that if a polynomial map $P:\mathbb{C}^N \rightarrow \mathbb{C}^N$
	is injective, then it must be surjective. Furthermore, a proof in \cite{rud95} shows that the inverse must be given by a polynomial via techniques from
	Galois theory. We prove a free Ax-Grothendieck theorem as Theorem \ref{faxgro}.
	
	The following is an immediate corollary of our inverse function theorem combined with the results in the preceeding paragraphs.
	\begin{thm}[Free Jacobian conjecture]\label{freejc}
		Let $\mathcal{M}(\mathbb{C})^N$ be the set of matrix $N$-tuples.
		Suppose $P:\mathcal{M}(\mathbb{C})^N \rightarrow \mathcal{M}(\mathbb{C})^N$ is a free polynomial map. The following are equivalent:
		\begin{enumerate}
			\item $DP(X)$ is a nonsingular map for each $X.$
			\item $P$ is injective.
			\item $P$ is bijective.
			\item $P^{-1}$ exists and $P^{-1}|_{\mathcal{M}_n(\mathbb{C})^N}$ agrees with free polynomial.
		\end{enumerate}
	\end{thm}
	We prove this result in Section \ref{alggeo}.
	
	We remark the last condition could be conjectured to be that $P^{-1}$ is a bona fide free polynomial. However, we caution that
	degree bounds in the Ax-Grothendieck theorem are very large \cite{bcw82}, and the theory of polynomial identity rings
	supplies many low degree polynomial identities satisfied by all matrix tuples of a specific size \cite{al50}. This means that there
	are a plethora of maps that have polynomial formulas for each specific size of matrix, but are not actually given by some polynomial.
	
	Additionally, we immediately obtain a matrix version of the commutative Jacobian conjecture
	via another application of the inverse function theorem.
	In this case, we do obtain true polynomials for the inverse map,
	because commutative polynomials are determined by their values on the scalars.
	\begin{thm}[Commuting matrix Jacobian conjecture]\label{jcmatrix}
		Let $P: \mathbb{C}^N \rightarrow \mathbb{C}^N.$
		The following are equivalent:
		\begin{enumerate}
			\item $DP(X)$ is an nonsingular map for each commuting matrix $N$-tuple $X.$
			\item $P$ is injective.
			\item $P$ is bijective.
			\item $P^{-1}$ exists and $P^{-1}$ is given by a polynomial.
		\end{enumerate}
	\end{thm}
	We prove this result in Section \ref{alggeo}.
	
	We caution that the structure of free maps greatly simplifies their geometry. The Jacobian conjecture in the classical context is notoriously difficult, but
	in the matricial context is shown here to be tractable.
	Indeed, free maps have generally been observed to encode nonlocal data in many contexts which is in strong contrast to the classical case.
	For another example, compare between Putinar's Positivstellensatz \cite{put93} to Helton's noncommutative Positivstellensatz \cite{helton02}.
	
\subsection{Some examples of domains of invertibility}
	We briefly give some examples of applications of our main result, the inverse function theorem for free maps.

	\subsubsection{Domains of inveribility for squaring}
	Take the function
		$$f(X) = X^2.$$
	Suppose we want to find a domain where $f$ is invertible.
	We obtain a derivative for $f$ given by the formula
		$$Df(X)[H]= XH+HX.$$
	Thus, by the inverse function theorem, we need the equation
		$$XH+HX=0$$
	to have no nontrivial solutions for each $X$ in our domain.
	This is a degenerate form of the famous Sylvester equation, so this will be nonsingular if $X$ has
	no eigenvalues in common with $-X.$ For a detailed account of the Sylvester equation, see Horn and Johnson \cite{horjoh85}.
	Thus, if we take a subset $H\subset \mathbb{C}$ such that $H \cap -H=\emptyset,$ and lift to the set of matrices
	with spectrum in $H,$ then $f$ will be invertible there. In fact, these are all possible maximal domains for such an inverse.
	
	\subsubsection{The quadratic symmetrization map}
	Consider,
		$$f(X,Y) = (X + Y, X^2 + Y^2).$$
	Taking the derivative,
		$$Df(X,Y)[H,K] = (H+K,HX +XH+KY+YK).$$
	So we need to check the second coordinate of the derivative is nonzero when $H = - K.$
	So, we want
		$$H(X-Y) +(X-Y)H= 0$$
	to have no nontrivial solutions.
	By the same use of Sylvesters equation as in the first example,
	this exactly says $X-Y$ needs to have spectrum disjoint from $Y-X.$
	
	\subsubsection{A more exotic quadratic map}
	Now consider the function
		$$f(X) = (X+X^2+[X,Y],Y+[X,Y]).$$
	Taking the derivative,
		$$Df(X,Y)[H,K] = (H+HX+HX+[H,Y]+[X,K],K+[H,Y]+[X,K]).$$
	Suppose this had a nontrivial solution at some $(X,Y)$ for $(H,K).$
	Either $\|H\|\geq\|K\|$ or $\|K\|\geq\|H\|.$
	In the case where $\|H\|\geq\|K\|,$ $H\neq 0,$ and
		$$\|H+HX+XH+[H,Y]+[X,K]\|\geq  \|H\|(1 - 4\|X\| - 2\|Y\|).$$
	So, it must be that $1 - 4\|X\| - 2\|Y\|\leq 0.$
	In the case where $\|K\|\geq\|H\|,$ $K\neq 0,$ and
		$$\|K+[H,Y]+[X,K]\| \geq \|K\|(1 - 2\|X\| - 2\|Y\|)$$
	So, it must be that $1 - 2\|X\| - 2\|Y\|\leq 0.$
	
	Restricting $f$ to the set of $(X,Y)$ such that $4\|X\| + 2\|Y\|<1$ precludes $Df$ from being singular.
	However, this fails to be a free domain since it is not closed with repect to direct sums. (See Section \ref{freeanalysis} for a formal
	definition of a free domain.) However, if we restrict $f$ to the set of $(X,Y)$ such that $\|X\|<\frac{1}{8}$ and $\|Y\|<\frac{1}{4}$
	we do indeed obtain a free map, and thus by the inverse function theorem, the function $f$ will be invertible there.
	
\section{Free analysis}\label{freeanalysis}
	
	Let $\mathcal{M}_n$ be the $n \times n$ matrices.
	We denote $\mathcal{M}^N = \bigcup \mathcal{M}_n^N.$
	
	A \emph{free set} $D \subset \mathcal{M}^N$ is closed under direct sums and joint similarity. That is, 
	\begin{enumerate}
		\item $A, B \in D \Rightarrow A \oplus B \in D,$
		\item $A \in D \cap \mathcal{M}_n^N, S \in GL_n \Rightarrow S^{-1}AS \in D.$
	\end{enumerate}
	Where
		$$(A_1,A_2,\ldots,A_N) \oplus (B_1,B_2,\ldots,B_N)=(A_1 \oplus B_1,A_2 \oplus B_2,\ldots,A_N\oplus B_N),$$
	and
		$$S^{-1}(A_1,A_2,\ldots,A_N)S=(S^{-1}A_1S,S^{-1}A_2S,\ldots,S^{-1}A_NS).$$
	A prototypical example of such a set is the zero set of some free polynomial map.
	For example, the commuting tuples of matrices form a free set.
	
	We define a \emph{free domain} $D \subset \mathcal{M}^N$ to be either a free set or,
	if working over a local field, a set that is relatively open in its orbit
	under conjugation by invertible matrices and is closed under direct sums.
	Any function on a free domain extends to a function on a free set
	as in the envelope method described in \cite{vvw}.
	
	A \emph{free map} $f:D \rightarrow \mathcal{M}^{\hat{N}}$ obeys the following
	\begin{enumerate}
		\item $f(A \oplus B) = f(A) \oplus f(B),$
		\item $f(S^{-1}AS) = S^{-1}f(A)S,$
		\item $D$ is a free domain.
	\end{enumerate}
	This definition of free sets and free maps is a direct generalization of the definition given in Helton-Klep-McCullough \cite{hkm11};
	in the language of Agler-McCarthy \cite{amy13} this generalizes functions on basic open sets in their free topology.
	
	Essentially, these maps are an emulation of the classical functional calculus for non-commuting tuples of operators.
	
	\emph{We note that we do not specify a ground ring for the matrices in the general inverse function theorem;
	it merely needs to have a multiplicative identity.}
	
	\subsection{Derivatives in free analysis}
	%ignore the implementation of $\oplus.$ That is, we can implement
	Helton, Klep and McCullough differentiated free maps in \cite{hkm11}. They obtained the formula
	\begin{equation}\label{hkmeqn}
		f\left(
		\begin{matrix}
		X & H \\
		0 & X \\
		\end{matrix}
		\right)
		=
		\left(
		\begin{matrix}
		f(X) & Df(X)[H] \\
		0 & f(X) \\
		\end{matrix}
		\right).
	\end{equation}
	These types of formulas are pervasive throughout the free analysis literature. For other references see Voiculescu \cite{voi10} and
	the recently completed tome by D. S. Kaliuzhnyi-Verbovetskyi and V. Vinnikov \cite{vvw}.
	
	%We note that, for free monomials, and, thus, for free polynomials by linearity,
	Formula \ref{hkmeqn} can be seen to be true for formal differentiation satisfying Leibniz rule.
	Thus, we eschew any analytic means for obtaining the derivative, and instead use the above as our definition. 
	That is, our results formally hold over any field, or indeed unital ring, and for sets that may have cusps or other exotic geometric features.
	We formalize the above in the following proposition.
	\begin{prop}
	Define
		$$D(X_i)[H_j] = \left\{\begin{matrix}
			H_i, & \text{if }i=j \\
			0, & \text{if }i\neq j \\
		\end{matrix}\right.,$$
	and require
		$$D(f+g)[H] = D(f)[H] + D(g)[H],$$
	and
		$$D(fg)[H] = D(f)[H]g + fD(g)[H].$$
	Equation \ref{hkmeqn} is satisfied.
	\end{prop}
	\begin{proof}
	We only need to prove this fact on monomials.
	This is obtained inductively via the following algebra: $$
		\begin{matrix}
			\left(
				\begin{matrix}
				X_i & H_i \\
				0 & X_i \\
				\end{matrix}
			\right)
			m\left(
				\begin{matrix}
				X & H \\
				0 & X \\
				\end{matrix}
			\right)
			& = & 
			\left(
				\begin{matrix}
				X_i & H_i \\
				0 & X_i \\
				\end{matrix}
			\right)
			\left(
				\begin{matrix}
				m(X) & Dm(X)[H] \\
				0 & m(X) \\
				\end{matrix}
			\right) \\ \\
			& = &
			\left(
				\begin{matrix}
				X_im(X) & X_iDm(X)[H] + H_im(X) \\
				0 & X_im(X) \\
				\end{matrix}
			\right) \\ \\
			& = &
			\left(
				\begin{matrix}
				X_im(X) & D(X_im(X))[H] \\
				0 & X_im(X) \\
				\end{matrix}
			\right).
		\end{matrix}
		$$
		\end{proof}

\section{The inverse function theorem}\label{ifts}
	We now prove the inverse function theorem.
	\begin{proof}
		$(\neg 1 \Rightarrow \neg 2)$
		Suppose $Df$ is singular at some $X$ and some direction $H\neq 0.$
		That is, $Df(X)[H] = 0$.
		So, applying \ref{hkmeqn} and the direct sum formula,
		$$
		\begin{matrix}
			f\left(
			\begin{matrix}
			X & H \\
			0 & X \\
			\end{matrix}
			\right)
	 	& = &
			\left(
			\begin{matrix}
			f(X) & Df(X)[H] \\
			0 & f(X) \\
			\end{matrix}
			\right) \\ \\
		& = &
			\left(
			\begin{matrix}
			f(X) & 0 \\
			0 & f(X) \\
			\end{matrix}
			\right)\\ \\
		& = &
			f\left(
			\begin{matrix}
			X & 0 \\
			0 & X \\
			\end{matrix}
			\right)
		\end{matrix}$$
		This equality witnesses the noninjectivity of $f.$
		
		$(1 \Rightarrow 2)$
		Suppose $Df(X)$ is not singular at any $X$.
		Let $X_1, X_2$ be two matrices of the same size such that $f(X_1) = f(X_2).$
		Let
		$$ $$
		$$S =
		\left(
		\begin{matrix}
		1 & 0 & 0 & 1 \\
		0   & 1 & 0 & 0\\
		0 & 0 & 1 & 0\\
		0 & 0 & 0 & 1  \\
		\end{matrix}
		\right).
		$$ 
		$$ $$
		Let,
		$$ $$
		$$
		X = 
		\left(
		\begin{matrix}
		X_1 & 0 & 0 & 0 \\
		0   & X_2 & 0 & 0\\
		0 & 0 & X_1 & 0\\
		0 & 0 & 0 & X_2  \\
		\end{matrix}\right).
		$$
		$$ $$
		Note,
			$$ $$
			$$S^{-1}XS=
			\left(
			\begin{matrix}
				X_1 & 0 & 0 & X_1-X_2 \\
				0   & X_2 & 0 & 0\\
				0 & 0 & X_1 & 0\\
				0 & 0 & 0 & X_2  \\
			\end{matrix}
			\right),
			$$
			$$ $$
		and, since
			$$ $$
			$$
			f(X)=
			\left(
				\begin{matrix}
				f(X_1) & 0 & 0 & 0 \\
				0   & f(X_2) & 0 & 0\\
				0 & 0 & f(X_1) & 0\\
				0 & 0 & 0 & f(X_2)  \\
				\end{matrix}\right)
			$$
			$$ $$
		because $f$ preserves direct sums, we obtain the fomula
		$$ $$
			$$S^{-1}f(X)S=
			\left(
			\begin{matrix}
				f(X_1) & 0 & 0 & f(X_1)-(X_2) \\
				0   & f(X_2) & 0 & 0\\
				0 & 0 & f(X_1) & 0\\
				0 & 0 & 0 & f(X_2)  \\
			\end{matrix}
			\right).
			$$
		$$ $$
		$$ $$
		So, via the similarity relation for free maps,
		$$ $$
		$$f\left(
		\begin{matrix}
		X_1 & 0 & 0 & X_1-X_2 \\
		0   & X_2 & 0 & 0\\
		0 & 0 & X_1 & 0\\
		0 & 0 & 0 & X_2  \\
		\end{matrix}
		\right)
		=
		\left(
		\begin{matrix}
		f(X_1) & 0 & 0 & f(X_1)-f(X_2) \\
		0   & f(X_2) & 0 & 0\\
		0 & 0 & f(X_1) & 0\\
		0 & 0 & 0 & f(X_2)  \\
		\end{matrix}
		\right).
		$$
		$$ $$
		$$ $$
		Thus, since we assumed $f(X_1)=f(X_2),$
			$$ $$
			$$
			f\left(
				\begin{matrix}
				X_1 & 0 & 0 & X_1-X_2 \\
				0   & X_2 & 0 & 0\\
				0 & 0 & X_1 & 0\\
				0 & 0 & 0 & X_2  \\
				\end{matrix}
			\right)
			=
			\left(
			\begin{matrix}
			f(X_1) & 0 & 0 & 0 \\
			0   & f(X_2) & 0 & 0\\
			0 & 0 & f(X_1) & 0\\
			0 & 0 & 0 & f(X_2)  \\
			\end{matrix}
			\right).
			$$
			$$ $$
		On the other hand, by \ref{hkmeqn},
		$$ $$
		$$f\left(
		\begin{matrix}
		X_1 & 0 & 0 & X_1-X_2 \\
		0   & X_2 & 0 & 0\\
		0 & 0 & X_1 & 0\\
		0 & 0 & 0 & X_2  \\
		\end{matrix}
		\right)
		=
		\left(
		\begin{matrix}
		f\left(\begin{matrix}X_1 & 0 \\0 & X_2 \end{matrix} \right)
		& Df\left(\begin{matrix}X_1 & 0 \\0 & X_2 \end{matrix} \right)\left[\begin{matrix}0 & X_1 - X_2 \\0 & 0 \end{matrix} \right] \\
		0 & f\left(\begin{matrix}X_1 & 0 \\0 & X_2 \end{matrix} \right) \\
		\end{matrix}
		\right).$$
		$$ $$
		So, $$Df\left(\begin{matrix}X_1 & 0 \\0 & X_2 \end{matrix} \right)\left[\begin{matrix}0 & X_1 - X_2 \\0 & 0 \end{matrix} \right] =
		\left(\begin{matrix}0 & 0 \\0 & 0 \end{matrix} \right).$$
		Thus,  $X_1-X_2 = 0,$ since we assumed $Df(X)[H]$ is nonsingular for all $X$ in the domain, or equivalently that $Df(X)[H]=0$ implies $H=0.$
		So, $f$ is injective.
		
		$(2 \Leftrightarrow 3)$ We leave this to the reader. It is similar to a proof in Helton-Klep-McCullough \cite{hkm11}.
	\end{proof}

\section{Polynomial maps} \label{alggeo}
	In this section we recall some classical results from algebraic geometry which we will use to prove the Jacobian
	conjecture for free polynomials and commuting matrix polynomials, and subsequently give these proofs.
	
	Ott-Heinrich Keller infamously suggested the following conjecture.
	\begin{ques}[Jacobian conjecture]
		Let $P: \mathbb{C}^N \rightarrow \mathbb{C}^N$ be a polynomial map.
		If $DP(x)$ is nonsingular for every $x$, must $P$ necesarily be invertible? Furthermore,
		can the inverse be taken to be a polynomial?
	\end{ques}

	James Ax and Alexander Grothendieck independently proved the following seemingly related result about polynomial maps.
	\begin{thm}[Ax-Grothendieck theorem \cite{ax68}, \cite{ega}]
		Let $P: \mathbb{C}^N \rightarrow \mathbb{C}^N$ be a polynomial map.
		If $P$ is injective, then $P$ is surjective.
	\end{thm}
	This reduces the Jacobian conjecture to showing the condition on the derivative implying global injectivity.
	Furthermore, this result has been refined to the following.
	\begin{thm}[Ax-Grothendieck theorem\cite{rud95}]
		Let $P: \mathbb{C}^N \rightarrow \mathbb{C}^N$ be a polynomial map.
		If $P$ is injective, then $P$ is surjective and $P^{-1}$ is given by a polynomial.
	\end{thm}
	This tacitly gives an equivalence between a polynomial map being invertible and having a polynomial inverse.
	
	This can be lifted to the free case.
	\begin{thm}[Free Ax-Grothendieck theorem]\label{faxgro}
		Let $P: \mathcal{M}(\mathbb{C})^N \rightarrow \mathcal{M}(\mathbb{C})^N$ be a free polynomial map.
		If $P$ is injective, then $P$ is surjective and $P^{-1}|_{\mathcal{M}^N_n}$ is given by a free polynomial.
	\end{thm}
	\begin{proof}
		For each size of matrix $n$, we view $P$ as a tuple of $dn^2$ commuting polynomials by replacing the free variables with
		indeterminant $n$ by $n$ matices.
		Since $P|_{\mathcal{M}_n(\mathbb{C})^N}$ is injective,
		$P^{-1}|_{\mathcal{M}_n(\mathbb{C})^N}$ is given by a tuple of $dn^2$ commuting polynomials.
		Since the global $P^{-1}$ is a free map by the inverse function theorem and continuous free maps are analytic,
		they have a power series of free polynomials\cite{vvw},
		the restriction
		$P^{-1}|_{\mathcal{M}_n(\mathbb{C})^N}$ must agree with a free polynomial;
		the terms in the power series for $P^{-1}$ must eventually vanish on all of $\mathcal{M}_n(\mathbb{C})^N$.
	\end{proof}

	We now prove the the Jacobian conjecture for free polynomials, Theorem \ref{freejc}.
	\begin{proof}[Proof of Theorem \ref{freejc}]
		$(1 \Leftrightarrow 2)$ follows from the inverse function theorem.
		$(2 \Leftrightarrow 3 \Leftrightarrow 4)$ is the free Ax-Grothendieck theorem.
	\end{proof}
	
	We now prove the the Jacobian conjecture for commuting matrix polynomials, Theorem \ref{jcmatrix}.
	\begin{proof}[Proof of Theorem \ref{freejc}]
		$(1 \Leftrightarrow 2)$ follows from the inverse function theorem.
		
		$(2 \Rightarrow 4)$
		The function $P|_{\mathcal{M}_1(\mathbb{C})^N}$ has a polynomial inverse $P^{-1}$ by the Ax-Grothendieck theorem.
		Since $P|_{\mathcal{M}_1(\mathbb{C})^N}$ is equal to $P$ as a polynomial, we indeed obtained global inverse.
		(The values on the scalars determine a commutative polynomial.)
		
		$(4 \Rightarrow 3 \Rightarrow 2)$ is trivial.
	\end{proof}

\bibliography{references}
\end{document}